\documentclass[a4paper,11pt]{article}
\usepackage[margin=2.5cm]{geometry}

\usepackage{amsfonts}
\usepackage{amsmath}
\usepackage{amsthm}
\usepackage{amssymb}
\usepackage{cleveref}
\usepackage{graphicx}
\usepackage{epstopdf}
\usepackage{algorithm}
\usepackage{algpseudocode}
\usepackage{bm}
\usepackage{url}
\ifpdf
  \DeclareGraphicsExtensions{.eps,.pdf,.png,.jpg}
\else
  \DeclareGraphicsExtensions{.eps}
\fi

\usepackage{pgfplots,pgfplotstable}
\pgfplotsset{compat=newest}
\usepackage{tikz}
\usetikzlibrary{shapes,arrows,snakes,mindmap,patterns}
\tikzstyle{decision} = [diamond, draw, fill=blue!20, text badly centered, node distance=3cm, inner sep=0pt]
\tikzstyle{sdecision} = [diamond, draw, fill=blue!20, text centered, node distance=3cm, inner sep=0pt, text width=3cm]

\tikzstyle{block} = [rectangle, draw, fill=blue!20, text centered, rounded corners]
\tikzstyle{sblock} = [rectangle, draw, fill=blue!20, text centered, rounded corners, text width = 6cm]

\tikzstyle{line} = [draw, -latex']
\tikzstyle{cloud} = [draw, ellipse,fill=red!20, node distance=3cm,
    minimum height=2em]
\tikzstyle{scloud} = [draw, ellipse,fill=red!20, node distance=3cm,
    minimum height=2em,text width = 5 cm]

\definecolor{blau0} {RGB}{ 131 198 216} 
\definecolor{blau}  {RGB}{  0  84 159}
\definecolor{grau}  {RGB}{  0  84 159}
\definecolor{grun}  {RGB}{  0 160   0}
\definecolor{rot}   {RGB}{204   7  30}
\definecolor{blau2} {RGB}{  0  61 128}
\definecolor{grun2} {RGB}{  0 85   0}
\definecolor{rot2}  {RGB}{120   7  30}
\definecolor{gelb}  {RGB}{70 70 70}

\newtheorem{theorem}{Theorem}
\newtheorem{lemma}[theorem]{Lemma}
\newtheorem{remark}[theorem]{Remark}

\numberwithin{equation}{section}
\numberwithin{theorem}{section}

\newcommand{\llangle}{\left\langle}
\newcommand{\rrangle}{\right\rangle}

\title{On the rate of convergence of alternating minimization for non-smooth non-strongly convex optimization in Banach spaces 
}

\author{Jakub W.\ Both\thanks{Department of Mathematics, University of Bergen, Bergen, Norway; $\{$\texttt{jakub.both@uib.no}$\}$}}

\date{}

\begin{document}

\maketitle

\begin{abstract}
  In this paper, the convergence of alternating minimization is established for non-smooth convex optimization in Banach spaces, and novel rates of convergence are provided. As objective function a composition of a smooth and a non-smooth part is considered with the latter being block-separable, e.g., corresponding to convex constraints or regularization. For the smooth part, three different relaxations of strong convexity are considered: (i) quasi-strong convexity; (ii) quadratic functional growth; and (iii) plain convexity. Linear convergence is established for the first two cases, generalizing and improving previous results for strongly convex problems; sublinear convergence is established for the third case, also improving previous results from the literature. All the convergence results have in common, that opposing to previous corresponding results for the general block coordinate descent, the performance of the alternating minimization is beneficially governed by properties of the single blocks, instead of global properties. Ultimately, not only the better conditioned block determines the performance, as has been similarly observed in the literature. But also the worse conditioned problem enhances the performance additionally, resulting in potentially significantly improved convergence rates. Furthermore, by solely using the convexity and smoothness properties of the problem, the results immediately apply in general Banach spaces.\\[0.5cm]
\end{abstract}

\section{Introduction}
The (cyclic) block coordinate descent, also referred to in the literature as non-linear block Gauss-Seidel or successive subspace correction method, is a classical and fundamental optimization algorithm~\cite{Ortega1970,Bertsekas1999}. Given a minimization problem with a block structure, it consists of the successive (exact) minimization with respect to the single blocks. Since numerous applications naturally inherit a block structure, block coordinate descent algorithms have been of great interest for decades, including variations as random coordinate descents and the successive inexact minimization, based, e.g., on the projected gradient descent or proximal point minimization. For an overview, we refer to the review paper~\cite{Wright2015}.
 
The convergence of the block coordinate descent has been extensively studied under various convexity and smoothness properties of the objective function, typically in Euclidean spaces. For instance, convergence of the algorithm has been established for non-smooth strongly convex optimization by Auslender~\cite{Auslender1976}, and for various sets of convexity assumptions as, e.g., quasi-convexity with respect to each block by Grippo and Sciandrone~\cite{Grippo1999,Grippo2000}. Furthermore, Bertsekas~\cite{Bertsekas1999} showed that any accumulation point of the sequence generated by the method is a stationary point if the successive minimization with respect to each block is well-defined. In the context of domain decomposition methods, Tai and Espedal~\cite{Tai1998} established a linear rate of convergence for the multiplicative Schwarz method applied to smooth strongly convex problems, which can ultimately be identified as the block coordinate descent method. Luo and Tseng~\cite{Luo1993} showed a linear rate of convergence for feasible descent methods under the existence of a local error bound of the objective function (generalizing strong convexity), and a proper separation of isocost surfaces; the block coordinate descent method satisfies the feasible descent property, e.g., for block coordinatewise strongly convex functions. Lately, Necoara et al.~\cite{Necoara2019} identified the class of of smooth convex functions satisfying a quadratic functional growth property as the largest class of functions for which feasible descent methods converge linearly, and provide linear convergence rates.

In this paper, we focus on the application of the block coordinate descent to the two-block structured model problem
\begin{align}
\label{model-problem-intro}
 \mathrm{min}  \left\{ H(x_1,x_2) \equiv f(x_1,x_2) + g_1(y) + g_2(z) \, \big| \, (x_1,x_2)\in \mathcal{B}_1 \times \mathcal{B}_2 \right\},
\end{align}
where $\mathcal{B}_1,\mathcal{B}_2$ are Banach spaces, $f$ is convex and smooth, and $g_1,g_2$ are non-smooth but convex, allowing, e.g., for block-separable constraints or non-smooth regularization (detailed properties are specified in \cref{section:problem}). In the case of just two blocks, the block coordinate descent is also often referred to as \textit{alternating minimization} (cf.~\cref{section:am}). Since problems of type~\eqref{model-problem-intro} are present in various applications, alternating minimization is widely used as a decoupling technique -- in particular if it is much more convenient or feasible to solve the corresponding subproblems instead of the globally coupled problem. This is has been for instance exploited in the context of the iteratively reweighted least squares method~\cite{Beck2015}, or the block partitioned solution of coupled partial differential equations, cf., e.g.,~\cite{Both2019}.

The presence of just two blocks allows for an improved convergence analysis of the alternating minimization compared to the general block coordinate descent. For (unconstrained) smooth strongly convex optimization, i.e., smooth and strongly convex $f$, and $g_1\equiv g_2 \equiv 0$, linear convergence has been established by Beck and Tetruashvili~\cite{Beck2013} with the rate just depending on the convexity modulus and the minimum of the Lipschitz constant of the partial derivatives, instead of a global one. Also for a smooth (simply) convex $f$, and convex $g_1,g_2$, sublinear convergence has been proved by Beck~\cite{Beck2015}. Again the multiplicative constant has been showed to only depend on the minimum of the Lipschitz constants of the partial derivatives, instead of a global one. The aforementioned results are constrained to (finite-dimensional) Euclidean spaces equipped with the $l_2$ norm. The proofs essentially utilize knowledge on first-order gradient descent methods as the (proximal) block coordinate descent. To our best knowledge, those results are the finest theoretical convergence results in the literature for alternating minimization.

In practice, strong convexity is a rather strong requirement, which often is not met. It is natural to ask whether linear convergence can be also achieved under more relaxed conditions. This is in particular motivated by the work of Necoara~\cite{Necoara2019}, in which linear convergence has been established for the general block gradient descent under various relaxations of strong convexity, including quasi-strong convexity and quadratic functional growth. To the author's knowledge, an improved analysis of the specific alternating minimization under such conditions has not been provided in the literature, yet.

Considering Banach spaces in the convergence analysis of the alternating minimization may have several benefits. We mention two. First, by allowing for describing smoothness and convexity properties of the problem with respect to problem-specific norms, sharper theoretical convergence rates may be derived -- even for problems in finite dimensional Euclidean spaces. And second, as already mentioned, alternating minimization can be applied for instance to develop robust splitting schemes for coupled partial differential equations. Abstract convergence results of the alternating minimization holding for infinitely dimensional Banach spaces may then allow for analyzing the convergence of the resulting scheme, independent of the discretization, cf., e.g.~\cite{Both2019} in the context of the Biot equations describing flow in deformable porous media.

In this work, we aim at complementing previous results on the convergence of the fundamental alternating minimization and generalize those for non-strongly convex problems of type~\eqref{model-problem-intro}. In this regard, our main contributions are convergence results for three different settings with a decreasing demand on the convexity properties of the problem:
\begin{itemize}
 \item Linear convergence assuming a quasi-strongly convex $f$, generalizing and improving the results in~\cite{Beck2013} for unconstrained smooth strongly convex optimization in Euclidean spaces. The final result is assessed numerically and by that demonstrated to be sharp.
 \item Linear convergence for convex $H$ with quadratic functional growth without the explicit need of feasible descent as, e.g., required by~\cite{Necoara2019}, cf.\ \cref{section:strongly-convex}.
 \item Sublinear convergence for convex $f$, generalizing and improving the results in~\cite{Beck2015} for convex optimization in Euclidean spaces, cf.\ \cref{section:convex}.
\end{itemize}
All results have in common, that opposing to corresponding results for the general block coordinate descent, the performance of the alternating minimization gains from properties of the separate single blocks instead of global properties. Furthermore, by solely utilizing the basic definitions of convexity and smoothness properties as well as the definition of the alternating minimization in the proofs, all results hold in Banach spaces.
To the best of the author's knowledge, all results are novel; in particular, the case of quadratic functional growth without feasible descent has not been studied in the literature, yet.

\section{The two-block structured model problem}\label{section:problem}
We consider the problem
\begin{align}
\label{model-problem}
 \mathrm{min}  \left\{ H(x_1,x_2) \equiv f(x_1,x_2) + g_1(y) + g_2(z) \, \big| \, (x_1,x_2)\in \mathcal{B}_1 \times \mathcal{B}_2 \right\},
\end{align}
where $\mathcal{B}_1,\mathcal{B}_2,f,g_1,g_2$ satisfy the following properties:
\begin{itemize}
 \item[(P1)] The feasible sets $\mathcal{B}_i$ are Banach spaces equipped with norms $\|\cdot\|_{i}$, $i=1,2$. Let $\mathcal{B}_i^\star$ denote the dual space of $\mathcal{B}_i$ equipped with the canonical dual norm $\| \cdot \|_{i,\star}$, and let $\llangle \cdot, \cdot \rrangle_i$ denote the duality product on $\mathcal{B}_i^\star \times \mathcal{B}_i$. If clear from the context, we omit specifying the index $i=1,2$ in duality pairings.

 \item[(P2)] The product space $\mathcal{B}_1 \times \mathcal{B}_2$ is equipped with a norm $\|\cdot\|$ and $\beta_1,\beta_2\geq 0$, satisfying for all $(x_1,x_2)\in \mathcal{B}_1 \times \mathcal{B}_2$
 \begin{subequations}
 \label{definition-beta}
 \begin{align}
  \|(x_1,x_2)\|^2 &\geq \beta_1\|x_1 \|_1^2,\\
  \|(x_1,x_2)\|^2 &\geq \beta_2\|x_2 \|_2^2,
 \end{align} 
 \end{subequations}
 and the duality pairing $\llangle \cdot, \cdot \rrangle$ satisfying for all $d_i\in\mathcal{B}_i^\star$ and $x_i\in\mathcal{B}_i$, $i=1,2$,
 \begin{align*}
  \llangle (d_1,d_2), (x_1,x_2) \rrangle &:= \llangle d_1, x_1 \rrangle + \llangle d_2, x_2 \rrangle.
 \end{align*}
 
 \item[(P3)] The functions $g_1 : \mathcal{B}_1 \rightarrow \mathbb{R} \cup \{ \infty\}$, $g_2 : \mathcal{B}_2 \rightarrow \mathbb{R} \cup \{ \infty\}$ are proper convex functions, which are (Fr\'echet) subdifferentiable on their domains, denoted by $\mathrm{dom}\, g_1$ and $\mathrm{dom}\,g_2$, respectively. Let their (Fr\'echet) subdifferentials be denoted by $\partial g_1$ and $\partial g_2$.
 
 \item[(P4)] The function $f:\mathcal{B}_1 \times \mathcal{B}_2 \rightarrow \mathbb{R}$ is convex and (Fr\'echet) differentiable over $\mathrm{dom}\,g_1 \times \mathrm{dom}\,g_2$. Let $\nabla f$ denote the (Fr\'echet) derivative of $f$.
 
 \item[(P5)] The partial (Fr\'echet) derivatives of $f$ with respect to the first and second components, denoted by $\nabla_1 f \in \mathcal{B}_1^\star$, $\nabla_2 f \in \mathcal{B}_2^\star$, respectively, are \textit{Lipschitz continuous} such that there exist 
 $L_1,L_2\in(0,\infty]$ satisfying
 \begin{align}
  \label{model-property-lipschitz-f-1a}
  \left\| \nabla_1 f(x_1 + h_1,x_2) - \nabla_1 f(x_1,x_2) \right\|_{1,\star} &\leq L_1 \| h_1 \|_1,\\
   \label{model-property-lipschitz-f-2a}
  \left\| \nabla_2 f(x_1 ,x_2+h_2) - \nabla_2 f(x_1,x_2) \right\|_{2,\star} &\leq L_2 \| h_2 \|_2,
 \end{align}
 or equivalently (by block versions of the so-called descent lemma~\cite{Bertsekas1999,Beck2015})
 \begin{align}
 \label{model-property-lipschitz-f-1}
  f(x_1+h_1,x_2) &\leq f(x_1,x_2) + \llangle \nabla_1 f(x_1,x_2),h_1 \rrangle + \frac{L_1}{2} \left\| h_1 \right\|_1^2,\\
  \label{model-property-lipschitz-f-2}
  f(x_1,x_2+h_2) &\leq f(x_1,x_2) + \llangle \nabla_2 f(x_1,x_2),h_2 \rrangle + \frac{L_2}{2} \left\| h_2 \right\|_2^2,
 \end{align}
 for all $(x_1,x_2)\in \mathrm{dom}\, g_1 \times \mathrm{dom}\, g_2$, $h_1 \in \mathcal{B}_1$, such that $x_1+h_1 \in \mathrm{dom}\, g_1$, and $h_2 \in \mathcal{B}_2$ such that $x_2+h_2 \in \mathrm{dom}\, g_2$. We explicitly assume that
 \begin{align*}
  \mathrm{min}\{L_1,L_2\} < \infty,
 \end{align*}
 i.e., we allow for non-Lipschitz continuity of one of the partial derivatives.

 \item[(P6)] The optimal set of the problem~\eqref{model-problem}, denoted by $X \subset \mathcal{B}_1 \times \mathcal{B}_2$ is non-empty, and the corresponding optimal value is denoted by $H^\star$.
 
\end{itemize}

The model problem~\eqref{model-problem} satisfying the properties (P1)--(P6) covers a wide range of problems including, for instance, smooth convex optimization under block-separable convex constraints and non-smooth block-separable regularization. For examples, we refer to~\cite{Beck2015}.

We remark that the use of an arbitrary norm $\|\cdot\|$ (instead of, e.g., the canonical norm on product spaces $\|\cdot\|^2 = \|\cdot\|_1^2 + \| \cdot \|_2^2$), and the associated introduction of $\beta_1$ and $\beta_2$, cf.\ (P2), are going to significantly contribute in the development of improved convergence rates of the alternating minimization in~\cref{section:strongly-convex,section:quadratic,section:convex}.

\section{Alternating minimization}\label{section:am}

In the following, we focus on the iterative solution of problem~\eqref{model-problem} by the classical \textit{alternating minimization}, as described in \cref{algorithm:alternating-minimization}. In order to make the alternating minimization well-defined, we further assume on model problem~\eqref{model-problem}:
\begin{itemize}
 \item[(P7)] For any $(\tilde{x}_1,\tilde{x}_2)\in \mathrm{dom}\, g_1 \times \mathrm{dom}\, g_2$, the following problems have minimizers
 \begin{align*}
  \underset{x_1\in\mathcal{B}_1}{\mathrm{min}}\, H(x_1,\tilde{x}_2),\\
  \underset{x_2\in\mathcal{B}_2}{\mathrm{min}}\, H(\tilde{x}_1,x_2).
 \end{align*}
\end{itemize}

\begin{algorithm}[h!]
\caption{Alternating minimization}
\label{algorithm:alternating-minimization}

 
  \vspace{0.25em}
 
  \textbf{Initialization:} $x^{0} = (x_1^{0},x_2^{0}) \in \mathrm{dom}\,g_1 \times \mathrm{dom}\,g_2$, such that 
  \begin{align}
  \label{alternating-minimization-step-0}
   x_2^{0} \in \mathrm{arg\,min}\left\{ H(x_1^{0}, x_2) \, \big| \, x_2 \in \mathcal{B}_2 \right\}.
  \end{align}
  \textbf{General step:} For $k=0,1,...$, given $x^k \in \mathrm{dom}\,g_1 \times \mathrm{dom}\,g_2$, find $x^{k+1} \in \mathcal{B}_1 \times \mathcal{B}_2$ such that
  \begin{align}
  \label{alternating-minimization-step-1}
   x_1^{k+1} &\in \mathrm{argmin} \left\{ H(x_1,x_2^k) \, \big| \, x_1 \in \mathcal{B}_1 \right\},\\
   \label{alternating-minimization-step-2}
   x_2^{k+1} &\in \mathrm{argmin} \left\{ H(x_1^{k+1},x_2) \, \big| \, x_2 \in \mathcal{B}_2 \right\}.
  \end{align}  
  \textbf{Abbreviation:} For $k=0,1,...$, define $x^{k+1/2} := (x_1^{k+1},x_2^k)$, $H^k:=H(x_1^k,x_2^k)$, $H^{k+1/2}:=H(x_1^{k+1},x_2^k)$.
\end{algorithm}
We remark the partial optimality condition~\eqref{alternating-minimization-step-0} on the initial guess, which corresponds to one half step of the alternating minimization. As in~\cite{Beck2015}, this is simply chosen for convenience and the sake of simpler notation in the subsequent analysis -- in particular \cref{lemma:optimality-conditions}. In view of the following lemma, for an overview on sub\-differentials of convex functions in Banach spaces and optimality conditions, we refer to~\cite{Mordukhovich2006}.

\begin{lemma}[Optimality conditions for iterates of the alternating minimization]\label{lemma:optimality-conditions}
Let $\{x^k = (x_1^k,x_2^k)\}_{k\geq 0}$ denote the sequence generated by the alternating minimization, cf. \cref{algorithm:alternating-minimization}. Then for all $k\geq 0$, it holds $x_1^{k+1} \in \mathrm{dom}\, g_1$, $x_2^k \in \mathrm{dom}\,g_2$, and
\begin{alignat}{2}
\label{result:optimality-conditions-1}
g_1(x_1^{k+1}) - g_1(x_1) &\leq - \llangle \nabla_1 f(x_1^{k+1},x_2^k), x_1^{k+1} - x_1 \rrangle&\quad&\forall x_1\in \mathrm{dom}\,g_1,\\
\label{result:optimality-conditions-2}
g_2(x_2^k) - g_2(x_2) &\leq - \llangle \nabla_2 f(x_1^k, x_2^k), x_2^k - x_2 \rrangle&\quad&\forall x_2\in \mathrm{dom}\,g_2.
\end{alignat}
\end{lemma}
 
\begin{proof}
 By construction, the optimality condition corresponding to the first step of the alternating minimization, defining $\{x_1^{k}\}_{k\geq 1}$, cf.~\cref{alternating-minimization-step-1}, reads $x_1^{k+1}\in \mathrm{dom}\, g_1$ and $0 \in \nabla_1 f(x_1^{k+1},x_2^k)+ \partial g_1(x_1^{k+1})$ for all $k\geq 0$. Thus, $-\nabla_1 f(x_1^{k+1},x_2^k) \in \partial g_1(x_1^{k+1})$, which by definition of a subdifferential is equivalent with~\cref{result:optimality-conditions-1}.
 
 Analogously,~\cref{result:optimality-conditions-2} follows by considering the second substep of the alternating minimization~\eqref{alternating-minimization-step-2}, defining $\{x_2^k\}_{k\geq 1}$. For $k=0$,~\cref{result:optimality-conditions-2} follows by construction of the initial guess, cf.~\eqref{alternating-minimization-step-0}.
\end{proof}

Moreover, being identical with successive minimization, $H^k$ satisfies the monotonicity principle
\begin{align*}
 H^0 \geq H^{1/2} \geq H^1 \geq ... \geq H^k \geq H^{k+1/2} \geq H^{k+1} \geq ...\qquad\text{for all }k\in\mathbb{N}_0.
\end{align*}

\section{Quasi-strongly convex case}\label{section:strongly-convex}

In this section, linear convergence is established for the alternating minimization applied to model problem~\eqref{model-problem} under the additional assumption of quasi-strong convexity for the smooth part of $H$:
\begin{itemize}
 \item[(P8a)] The function $f:\mathcal{B}_1\times\mathcal{B}_2 \rightarrow \mathbb{R}$ is \textit{quasi-strongly convex} with modulus $\sigma>0$, i.e., 
  for all $x=(x_1,x_2)\in \mathrm{dom}\, g_1 \times \mathrm{dom}\, g_2$ and $\bar{x}=\mathrm{arg\,min}\left\{\|x-y\|\,\big| \, y\in X \right\}$ being the orthogonal projection of $x$ onto $X$, it holds that
 \begin{align*}
  f(\bar{x}) \geq f(x) + \llangle \nabla f(x), \bar{x} - x \rrangle + \frac{\sigma}{2} \|x - \bar{x}\|^2.
 \end{align*}
\end{itemize}
By the convexity of $g_1$ and $g_2$, $H$ inherits quasi-strong convexity from $f$. It is interesting to mention that a quasi-strongly convex function does not even require to be convex; however, any strongly convex function is clearly quasi-strongly convex.

The following result generalizes and improves a convergence result in~\cite{Beck2013}.

\begin{theorem}[Linear convergence under quasi-strong convexity]\label{theorem:convergence-strong-convexity}
 Assume that $\mathrm{(P1)}$--$\mathrm{(P7)}$ and $\mathrm{(P8a)}$ are satisfied. Let $\{x^k\}_{k\geq 0}$ be the sequence generated by the alternating minimization, cf.\ \cref{algorithm:alternating-minimization}, and $H^k:=H(x^k)$. Then  for all $k\geq 0$ it holds that
 \begin{align*}
  H^{k} - H^\star 
  \leq \left[\left(1 - \frac{\sigma\beta_1}{L_1} \right) \left(1 - \frac{\sigma\beta_2}{L_2} \right)\right]^k \left( H^{0} - H^\star \right).
 \end{align*}
 In the case of $\mathrm{max}\left\{\tfrac{L_1}{\beta_1},\tfrac{L_2}{\beta_2}\right\}=\infty$ this has to be understood as
\begin{align}
\label{final-result:strong-convex-2}
  H^{k} - H^\star 
  \leq \left(1 - \frac{\sigma}{\mathrm{min}\left\{\tfrac{L_1}{\beta_1},\tfrac{L_2}{\beta_2}\right\}} \right)^k \left( H^{0} - H^\star \right).
 \end{align}
\end{theorem}

\begin{proof} 
We consider the first half-step of the alternating minimization and show
\begin{align}
\label{proof:strong-convex-result-1}
 H^{k+1/2} - H^\star \leq \left(1 - \frac{\sigma\beta_1}{L_1} \right) \left( H^{k} - H^\star\right)\qquad\text{for all }k\in\mathbb{N}_0.
\end{align}
Without loss of generality we assume that $\tfrac{L_1}{\beta_1}<\infty$; note that~\cref{proof:strong-convex-result-1} holds immediately for $\tfrac{L_1}{\beta_1}=\infty$. In order to prove~\cref{proof:strong-convex-result-1}, we first utilize the quasi-strong convexity of $f$, the definition of $\beta_1$, cf.~\cref{definition-beta}, a simple rescaling, the fact that $\tfrac{\sigma\beta_1}{L_1}\in(0,1]$, and Lipschitz continuity of $\nabla_1 f$, cf.~\cref{model-property-lipschitz-f-1}. For this, let $\bar{x}^k=(\bar{x}_1^k,\bar{x}_2^k):= \mathrm{argmin}\left\{ \|x-x^k\| \,\big| \, x\in X \right\} \in \mathrm{dom}\, g_1 \times \mathrm{dom}\, g_2$ such that $H^\star = H(\bar{x}^k)$. Ultimately, it holds that
\begin{align}
\label{proof:strong-convex:aux-1}
 f(x^k) - f(\bar{x}^k)
 &\leq 
 \llangle \nabla f(x^k), x^k - \bar{x}^k \rrangle - \frac{\sigma}{2} \left\| x^k - \bar{x}^k \right\|^2\\
 \nonumber
 &\leq 
 \llangle \nabla f(x^k), x^k - \bar{x}^k \rrangle - \frac{\sigma\beta_1}{2} \left\| x^k_1 - \bar{x}^k_1 \right\|_1^2\\
 \nonumber
 &=
 \frac{L_1}{\sigma\beta_1} \left[ \llangle \nabla_1 f(x^k), \frac{\sigma\beta_1}{L_1} \left(x^k_1 - \bar{x}^k_1\right) \rrangle - \frac{L_1}{2} \left\| \frac{\sigma\beta_1}{L_1} \left(x^k_1 - \bar{x}^k_1\right) \right\|_1^2 \right]\\
 \nonumber
 &\qquad 
 +
 \llangle \nabla_2 f(x^k), x^k_2 - \bar{x}^k_2 \rrangle\\
 \nonumber
 &\leq 
 \frac{L_1}{\sigma\beta_1} \left[  f(x^k) - f\left(x_1^k+\frac{\sigma\beta_1}{L_1} \left( \bar{x}^k_1 - x_1^k \right), x_2^k\right)\right]\\
 \nonumber
 &\qquad 
 +
 \llangle \nabla_2 f(x^k), x^k_2 - \bar{x}^k_2 \rrangle.
\end{align}
Furthermore,~\cref{definition-beta}, \cref{model-property-lipschitz-f-1}, and (P8a) imply that  $\tfrac{\sigma\beta_1}{L_1}\in(0,1]$. Thus, by convexity of $g_1$ it holds that
\begin{align*}
 g_1\left(\frac{\sigma\beta_1}{L_1} \bar{x}^k_1 + \left(1 - \frac{\sigma\beta_1}{L_1} \right) x_1^k\right)
 \leq 
 \frac{\sigma\beta_1}{L_1} g_1(\bar{x}^k_1)
 +
 \left(1 - \frac{\sigma\beta_1}{L_1}\right) g_1(x_1^k).
\end{align*}
By reordering terms, we obtain
\begin{align}
\label{proof:strong-convex:aux-2}
 g_1(x_1^k) - g_1(\bar{x}^k_1) \leq \frac{L_1}{\sigma\beta_1} \left[ g_1(x_k) - g_1\left(x_1^k + \frac{\sigma\beta_1}{L_1} \left( \bar{x}^k_1 - x_1^k\right) \right)\right].
\end{align}
By \cref{lemma:optimality-conditions}, it holds that
\begin{align}
\label{proof:strong-convex:aux-3}
g_2(x_2^k) - g_2(\bar{x}^k_2) \leq - \llangle \nabla_2 f(x^k), x_2^k - \bar{x}^k_2 \rrangle.
\end{align}
By combining~\cref{proof:strong-convex:aux-1,proof:strong-convex:aux-2,proof:strong-convex:aux-3}, we obtain for the objective function
\begin{align*}
 H^k - H^\star
 &= 
 f(x^k) - f(\bar{x}^k) + g_1(x_1^k) - g_1(\bar{x}^k_1) + g_2(x_2^k) - g_2(\bar{x}^k_2) \\
 &\leq 
  \frac{L_1}{\sigma\beta_1} \left[  H^k - H\left(x_1^k+\frac{\sigma\beta_1}{L_1} \left( \bar{x}^k_1 - x_1^k \right), x_2^k\right)\right].
\end{align*}
By employing the optimality property of $x_1^{k+1}$, cf.~\cref{alternating-minimization-step-1}, we obtain
\begin{align*}
 H^k - H^\star
 &\leq 
  \frac{L_1}{\sigma\beta_1} \left(  H^k - H^{k+1/2} \right).
\end{align*}
Reordering terms finally yields~\cref{proof:strong-convex-result-1}.

By symmetry, it analogously follows that
\begin{align}
\label{proof:strong-convex-result-2}
 H^{k+1} - H^\star \leq \left(1 - \frac{\sigma\beta_2}{L_2} \right) \left( H^{k+1/2} - H^\star\right).
\end{align}
Hence, by combining~\cref{proof:strong-convex-result-1,proof:strong-convex-result-2}, we obtain
 \begin{align*}
  H^{k+1} - H^\star 
  \leq \left(1 - \frac{\sigma\beta_1}{L_1} \right) \left(1 - \frac{\sigma\beta_2}{L_2} \right) \left( H^{k} - H^\star \right),
 \end{align*}
and the assertion follows by induction.
\end{proof}

\subsection{Comparison to the literature}

Linear convergence of the general block coordinate descent (for an arbitrary number of blocks) has been previously established for strongly and non-strongly convex optimization. In the following, we recall some results from the literature for a comparison with the convergence result in \cref{theorem:convergence-strong-convexity}, specific for alternating minimization. All results have been derived for problems stated in Euclidean spaces, equipped with $l_2$ norms, such that $\beta_1=\beta_2=1$ in~\cref{definition-beta}.
 
For smooth strong\-ly convex optimization subject to  block-separable convex constraints, the general block coordinate descent for $N\geq 2$ number of blocks, has been previously showed to converge q-linearly~\cite{Luo1993,Wang2014}. In particular, let $L$ denote the global Lipschitz constant of $\nabla f$. Then for all $k\geq 0$ it holds that
\begin{align*}
  H^k - H^\star \leq \left(1 - \frac{\sigma^2}{\sigma^2 + 2\left(L(1+\sqrt{N}) + 2\right)\!\!\left(\sigma + (L+1)\left(L\sqrt{N}+2\right)\right)}\right)^k \!\!\left( H^0 - H^\star\right).
\end{align*}

Recently, linear convergence has been also established for smooth objective functions with quadratic functional growth (see also \cref{section:quadratic}). This includes strongly and quasi-strongly convex functions~\cite{Necoara2019}. In particular, for all $k\geq 0$ it holds that
\begin{align*}
  H^k - H^\star \leq  \left(1 - \frac{\sigma^2}{\sigma^2 + 4\left(3 + \sqrt{N} \right)^2 L^2} \right)^k \left( H^0 - H^\star\right).
\end{align*}

In the context of domain decomposition methods for PDEs, linear convergence of a general multiplicative Schwarz method has been established~\cite{Tai1998}, which includes alternating minimization. Considering unconstrained smooth strongly convex optimization with the objective function being three times differentiable, asymptotic linear convergence was proved with asymptotic rate $1 - \tfrac{\sigma^2}{\sigma^2 + 8 L^2}$.

Opposing to the results for the general block coordinate descent, in the special case of just two blocks, linear convergence with further improved convergence rates is guaranteed. For instance, in~\cite{Beck2013}, linear convergence of the alternating minimization has been established for unconstrained smooth strongly convex optimization, i.e., the model problem~\eqref{model-problem} under the simplification $g_1\equiv g_2\equiv 0$. The final convergence result is identical with~\cref{final-result:strong-convex-2}. Thereby, convergence is ensured already if just one partial derivative is Lipschitz continuous. 

After all, the convergence result in~\cref{theorem:convergence-strong-convexity}, provides three novel improvements compared to all mentioned results:
 \begin{itemize}
 
  \item[(i)] The theoretical convergence rate has a multiplicative (i.e., squared) character if $\mathrm{max}\{L_1,L_2\}<\infty$.
  
  \item[(ii)] Convergence is guaranteed for the general non-smooth quasi-strongly convex case with same rate as for the smooth strongly convex case.
  
  \item[(iii)] The result holds in general Banach spaces.
 \end{itemize}
 
\subsection{Numerical example}
In order to assess the sharpness of the convergence result in \cref{theorem:convergence-strong-convexity}, we present a simple numerical example. Representative as practical lower bound for the 'worst-case' theoretical convergence rate, we consider a problem within the overly favorable class of unconstrained smooth strongly convex quadratic optimization in Euclidean spaces:
\begin{align*}
 \mathrm{min}\left\{ H(\mathbf{x}_1,\mathbf{x}_2) \equiv \left.\frac{1}{2} \begin{bmatrix} \mathbf{x}_1 \\ \mathbf{x}_2 \end{bmatrix}^\top \begin{bmatrix} \mathbf{A} & \mathbf{B}^\top \\ \mathbf{B} & \mathbf{C} \end{bmatrix} \begin{bmatrix} \mathbf{x}_1 \\ \mathbf{x}_2 \end{bmatrix} - \begin{bmatrix} \mathbf{b}_1 \\ \mathbf{b}_2 \end{bmatrix}^\top \begin{bmatrix} \mathbf{x}_1 \\ \mathbf{x}_2 \end{bmatrix} \  \right| \ \mathbf{x}_1 \in \mathbb{R}^n,\ \mathbf{x}_2 \in \mathbb{R}^m
\right\}
\end{align*}
with 
\begin{align*}
 \mathbf{M} := \begin{bmatrix} \mathbf{A} & \mathbf{B}^\top \\ \mathbf{B} & \mathbf{C} \end{bmatrix}
\end{align*}
being symmetric positive definite. This problem clearly satisfies the assumptions of \cref{theorem:convergence-strong-convexity}, as smoothness and strong convexity are satisfied with respect to standard $l_2$ norms for the Euclidean spaces. In particular, it holds
$\beta_1=\beta_2=1$, $L_1=\lambda_\mathrm{max}\left(\mathbf{A}\right)$ and $L_2=\lambda_\mathrm{max}\left( \mathbf{C} \right)$, $\sigma=\lambda_\mathrm{min}(\mathbf{M})$, where $\lambda_\mathrm{max}(\cdot)$ and $\lambda_\mathrm{min}(\cdot)$ respectively denote the maximal and minimal eigenvalues. Ultimately, q-linear convergence is guaranteed such that for all $k\geq 0$ it holds
\begin{align*}
H^{k+1} - H^\star \leq \left(1 - \frac{\lambda_\mathrm{min}(\mathbf{M})}{\lambda_\mathrm{max}(\mathbf{A})} \right)\left(1 - \frac{\lambda_\mathrm{min}(\mathbf{M})}{\lambda_\mathrm{max}(\mathbf{C})} \right) \left(H^{k} - H^\star\right).
\end{align*}

By exploiting the generality of \cref{theorem:convergence-strong-convexity} and choosing problem-dependent norms $\|\cdot\|_1$, $\|\cdot\|_2$, and $\|\cdot\|$, the theoretical convergence rate can be significantly improved. For this, let $\|\cdot\|_1:=\|\cdot \|_{\mathbf{A}}$, $\|\cdot \|_2:=\| \cdot \|_\mathbf{C}$, and $\|\cdot \|:=\| \cdot \|_{\mathbf{M}}$, where for instance $\|\mathbf{x}_1 \|_\mathbf{A}^2:= \mathbf{x}_1^\top \mathbf{A} \mathbf{x}_1$. Finally, $H$ is strongly convex with respect to $\| \cdot \|$ with modulus $\sigma=1$, and the partial block derivatives $\nabla_1 H$ and $\nabla_2 H$ are Lipschitz continuous with Lipschitz constants $L_1=L_2=1$. Furthermore, since $\mathbf{M}$ is symmetric positive definite, so are $\mathbf{A}$ and $\mathbf{C}$, as well as the corresponding Schur complements $\mathbf{S}_\mathbf{A}:=\mathbf{A} - \mathbf{B}^\top \mathbf{C}^{-1} \mathbf{B}$ and $\mathbf{S}_\mathbf{C}:=\mathbf{C} - \mathbf{B} \mathbf{A}^{-1} \mathbf{B}^\top$, and it holds
\begin{alignat*}{3}
 \left\|(\mathbf{x}_1,\mathbf{x}_2) \right\|_\mathbf{M}^2 &&\geq \left\| \mathbf{x}_1 \right\|_{S_\mathbf{A}}^2 &&\geq \lambda_\mathrm{min}\left(\mathbf{A}^{-1} \mathbf{S}_\mathbf{A} \right) \left\| \mathbf{x}_1 \right\|_\mathbf{A}^2,\\
 \left\|(\mathbf{x}_1,\mathbf{x}_2) \right\|_\mathbf{M}^2 &&\geq \left\| \mathbf{x}_2 \right\|_{S_\mathbf{C}}^2 &&\geq \lambda_\mathrm{min}\left(\mathbf{C}^{-1} \mathbf{S}_\mathbf{C} \right) \left\| \mathbf{x}_2 \right\|_\mathbf{C}^2.
\end{alignat*}
Thus, $\beta_1 = \lambda_\mathrm{min}\left(\mathbf{A}^{-1}\mathbf{S}_\mathbf{A}\right)$ and $\beta_2=\lambda_\mathrm{min}\left(\mathbf{C}^{-1}\mathbf{S}_\mathbf{C}\right)$. Thereby, \cref{theorem:convergence-strong-convexity} predicts q-linear convergence such that for all $k\geq 0$ it holds
\begin{align}
\label{numerics:theoretical-rate}
H^{k+1} - H^\star \leq \underbrace{\left(1 - \lambda_\mathrm{min}\left(\mathbf{A}^{-1} \mathbf{S}_\mathbf{A} \right) \right)\left(1 - \lambda_\mathrm{min}\left(\mathbf{C}^{-1} \mathbf{S}_\mathbf{C} \right) \right)}_{=:\eta} \left(H^{k} - H^\star\right).
\end{align}

In the following, we verify the sharpness of the theoretical rate $\eta$, as predicted above, for a small example. Let $n=3$, $m=2$, and define
\begin{align*}
\mathbf{A}:= \begin{bmatrix} 
  5 & -1 & -2\\
 -1 & 6 & -2 \\
 -2 & -2 & 6 
 \end{bmatrix}\!\!,\, 
 \mathbf{B}:=\begin{bmatrix}
                1 & 0.5 & 0.2 \\
                -1 & 2 & 1              
             \end{bmatrix}\!\!,\,
 \mathbf{C}:=\begin{bmatrix}
              2 & 0.4 \\
              0.4 & 1.4
             \end{bmatrix}\!\!,
 \mathbf{b}_1:=\begin{bmatrix} 1\\ 1 \\ 1 \end{bmatrix}\!\!,\,
 \mathbf{b}_2:=\begin{bmatrix} 1\\ 1 \end{bmatrix}\!\!.
\end{align*}
For this choice, the theoretical rate as defined in~\cref{numerics:theoretical-rate} is given by $\eta\approx 0.7222$. The performance of the alternating minimization, cf.\ \cref{algorithm:alternating-minimization}, for the initial guess $\mathbf{x}_1^0:=\left[ 0, 0, 0\right]^\top$ is visualized in \cref{figure:numerics-performance}. In addition, the  theoretically predicted convergence in~\cref{numerics:theoretical-rate} is displayed as well. Ultimately, we observe a good agreement between the theoretical bound and the asymptotic practical convergence rate.

\begin{figure}[!h]
 \centering
 \begin{tikzpicture}
 \begin{semilogyaxis}[
    ylabel={error},
    width = 0.5\textwidth,
    xlabel={iterations [$k$]},
    legend entries={$H^k - H^\star$, $\mathcal{O}\left(\eta^k\right)$},
    legend cell align=left,
    legend columns=1,
    legend style={at={(0.95,0.95)}},
    xmin = 0,
    ymin = 0,
    ymajorgrids=true, major grid style={line width=.2pt, draw=gray!20},
    xmajorgrids=true
]

\addplot [solid, line width=1.25pt, color=grun]   coordinates {
(1,0.2827)
(2,0.0206)
(3,0.0073)
(4,0.0050)
(5,0.0036)
(6,0.0026)
(7,0.0019)
(8,0.0013)
(9,9.6927e-04)
(10,6.9995e-04)
(11,5.0547e-04)
(12,3.6502e-04)
(13,2.6360e-04)
(14,1.9036e-04)
(15,1.3747e-04)
(16,9.9276e-05)
(17,7.1693e-05)
(18,5.1771e-05)
(19,3.7388e-05)
(20,2.7000e-05)
(21,1.9499e-05)
(22,1.4082e-05)
(23,1.0169e-05)
(24,7.3438e-06)
(25,5.3038e-06)
(26,3.8307e-06)
(27,2.7664e-06)
(28,1.9977e-06)
(29,1.4426e-06)
(30,1.0417e-06)
(31,7.5230e-07)
}; 

\addplot [dashed, line width=0.75pt, color=black]   coordinates {
(1,0.005*0.722158700234744710400036638020537793636322021484375)
(31,0.005*0.722158700234744710400036638020537793636322021484375^31)
};

\end{semilogyaxis}
\end{tikzpicture}
 \caption{\label{figure:numerics-performance} Comparison of practical and theoretical convergence for the alternating minimization.}
\end{figure}
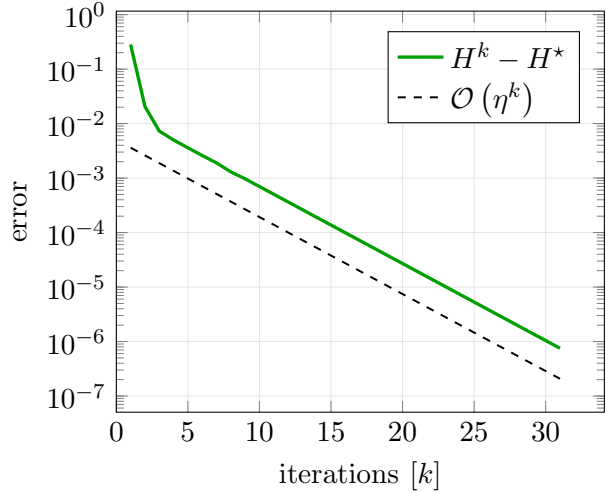

We conclude, that \cref{theorem:convergence-strong-convexity} allows for theoretically predicting sharp bounds of the practical convergence rate of the alternating minimization.

\section{Convex case with quadratic functional growth}\label{section:quadratic}

In this section, linear convergence is established for the alternating minimization applied to the model problem~\eqref{model-problem} with the objective function satisfying a quadratic functional growth property. We stress that opposing to the analysis of feasible descent methods, cf., e.g.,~\cite{Necoara2019}, a feasible descent property is not explicitly required. Such would be, e.g., ensured for block coordinatewise strongly convex objective functions. 

The property of quadratic functional growth reads:
\begin{itemize}

 \item[(P8b)] The objective function $H:\mathcal{B}_1\times\mathcal{B}_2 \rightarrow \mathbb{R}$ has \textit{quadratic functional growth} with modulus $\kappa>0$ with respect to the optimal set $X$, i.e., for all $x=(x_1,x_2)\in \mathrm{dom}\, g_1 \times \mathrm{dom}\, g_2$ and $\bar{x}:=\mathrm{arg\,min}\left\{\|x-y\|\,\big| \, y\in X \right\}$ being the orthogonal projection of $x$ onto $X$, it holds that
 \begin{align*}
  H(x) - H(\bar{x}) \geq \frac{\kappa}{2} \left\| x - \bar{x} \right\|^2.
 \end{align*}
\end{itemize}
\noindent
Quasi-strong convexity implies quadratic functional growth~\cite{Necoara2019}, but not vice versa; functions satisfying (P8b) do not even necessarily require to be convex~\cite{Zhang2015}.

The proof of the next result follows a similar strategy as the proof of \cref{theorem:convergence-strong-convexity}.

\begin{theorem}[Linear convergence under quadratic functional growth]\label{theorem:convergence-quadratic-growth-III}
 Assume that $\mathrm{(P1)}$--$\mathrm{(P7)}$ and $\mathrm{(P8b)}$ are satisfied. Let $\{x^k\}_{k\geq 0}$ be the sequence generated by the alternating minimization, cf.\ \cref{algorithm:alternating-minimization}, and $H^k:=H(x^k)$. Then for all $k\geq 0$ it holds that
 \begin{align*}
  H^{k} - H^\star 
  \leq \left[\left(1 - \frac{\kappa\beta_1}{8 L_1} \right) \left(1 - \frac{\kappa\beta_2}{8 L_2} \right)\right]^k \left( H^{0} - H^\star \right).
 \end{align*}
 In the case of $\mathrm{max}\left\{\tfrac{L_1}{\beta_1},\tfrac{L_2}{\beta_2}\right\}=\infty$, this has to be understood as
\begin{align*}
  H^{k} - H^\star 
  \leq \left(1 - \frac{\kappa}{8 \mathrm{min}\left\{\tfrac{L_1}{\beta_1},\tfrac{L_2}{\beta_2}\right\}} \right)^k \left( H^{0} - H^\star \right).
 \end{align*}
\end{theorem}

\begin{proof}
 We consider the first half-step of the alternating minimization and show
\begin{align}
\label{proof:quadratic-growth-III-result-1}
 H^{k+1/2} - H^\star \leq \left(1 - \frac{\kappa\beta_1}{8 L_1} \right) \left( H^{k} - H^\star\right).
\end{align}
Without loss of generality, we assume that $\tfrac{L_1}{\beta_1}<\infty$; note that~\cref{proof:quadratic-growth-III-result-1} holds immediately for $\tfrac{L_1}{\beta_1}=\infty$. In order to prove~\cref{proof:quadratic-growth-III-result-1}, we first utilize the convexity of $f$ yielding
\begin{align}
\label{proof:quadratic-growth-III:aux-1}
f(x^k) - f(\bar{x}^k)
  &\leq 
  \llangle \nabla_1 f(x_1^k, x_2^k), x_1^k - \bar{x}_1^k \rrangle +
  \llangle \nabla_2 f(x_1^k, x_2^k), x_2^k - \bar{x}_2^k \rrangle,
\end{align}
where $\bar{x}^k=(\bar{x}_1^k,\bar{x}_2^k):= \mathrm{argmin}\left\{ \|x-x^k\| \,\big| \, x\in X \right\} \in \mathrm{dom}\, g_1 \times \mathrm{dom}\, g_2$ such that $H^\star = H(\bar{x}^k)$. For $\gamma\in(0,1]$ to be specified later, using the Lipschitz continuity of $\nabla_1 f$, cf.~\cref{model-property-lipschitz-f-1a}, the convexity of $f$, and the definition of $\beta_1$, cf.~\cref{definition-beta}, we obtain
\begin{align}
\label{proof:quadratic-growth-III:aux-1b}
  &\llangle \nabla_1 f(x_1^k, x_2^k), x_1^k - \bar{x}_1^k \rrangle \\
  \nonumber
  &\quad=
   \llangle \nabla_1 f(x_1^k, x_2^k) - \nabla_1 f\left(x_1^k + \gamma \left(\bar{x}_1^k - x_1^k\right),x_2^k\right), x_1^k - \bar{x}_1^k \rrangle \\
   \nonumber
   &\qquad\qquad+ \frac{1}{\gamma}\llangle \nabla_1 f\left(x_1^k + \gamma \left(\bar{x}_1^k - x_1^k\right),x_2^k\right), \gamma \left(x_1^k - \bar{x}_1^k\right) \rrangle \\
   \nonumber
  &\quad\leq 
   L_1 \gamma \| x_1^k - \bar{x}_1^k \|_1^2 + \frac{1}{\gamma} \left[ f\left(x_1^k,x_2^k\right) - f\left(x_1^k + \gamma \left(\bar{x}_1^k - x_1^k\right), x_2^k \right) \right]\\
      \nonumber
  &\quad\leq 
   \frac{L_1}{\beta_1} \gamma \| x^k - \bar{x}^k \|^2 + \frac{1}{\gamma} \left[ f\left(x_1^k,x_2^k\right) - f\left(x_1^k + \gamma \left(\bar{x}_1^k - x_1^k\right), x_2^k \right) \right].
\end{align}
By convexity of $g_1$, it holds that
\begin{align}
\label{proof:quadratic-growth-III:aux-2}
 g_1(x_1^k) - g_1(\bar{x}^k_1) \leq \frac{1}{\gamma} \left[ g_1(x_k) - g_1\left(x_1^k + \gamma\left( \bar{x}_1^k - x_1^k\right) \right)\right].
\end{align}
By \cref{lemma:optimality-conditions}, it holds that
\begin{align}
\label{proof:quadratic-growth-III:aux-3}
g_2(x_2^k) - g_2(\bar{x}_2^k) \leq - \llangle \nabla_2 f(x^k), x_2^k - \bar{x}_2^k \rrangle.
\end{align}
By combining~\cref{proof:quadratic-growth-III:aux-1}, \cref{proof:quadratic-growth-III:aux-1b}, \cref{proof:quadratic-growth-III:aux-2}, \cref{proof:quadratic-growth-III:aux-3}, we obtain for the objective function
\begin{align}
\label{proof:quadratic-growth-III:aux-4}
 H^k - H^\star
 &= 
 f(x^k) - f(\bar{x}^k) + g_1(x_1^k) - g_1(\bar{x}_1^k) + g_2(x_2^k) - g_2(\bar{x}_2^k) \\
 \nonumber
 &\leq 
  \frac{L_1}{\beta_1} \gamma \| x^k - \bar{x}^k \|^2 
  + \frac{1}{\gamma} \left[ H(x_1^k,x_2^k) - H\left(x_1^k + \gamma \left( \bar{x}_1^k - x_1^k \right), x_2^k \right) \right].
\end{align}
Thus, by utilizing the quadratic growth of $H$ and the optimality property of $x_1^{k+1}$ based on the first step of the alternating minimization, cf.~\cref{alternating-minimization-step-1}, it follows for all $\gamma\in(0,1]$ that
\begin{align*}
 H^k - H^\star
 \leq 
  \frac{2 L_1 \gamma}{\kappa\beta_1} \left(  H^k - H^\star \right)
  +
  \frac{1}{\gamma} \left( H^k - H^{k+1/2} \right).
\end{align*}
By (optimally) choosing $\gamma=\frac{\kappa\beta_1}{4L_1}$, we obtain
\begin{align*}
 H^k - H^\star
 \leq 
  \frac{8 L_1}{\kappa\beta_1} \left( H^k - H^{k+1/2} \right),
\end{align*}
which finally yields~\cref{proof:quadratic-growth-III-result-1}, after reordering terms.

By symmetry, it analogously follows that
\begin{align}
\label{proof:quadratic-growth-III-result-2}
 H^{k+1} - H^\star \leq \left(1 - \frac{\kappa\beta_2}{8 L_2} \right) \left( H^{k+1/2} - H^\star\right).
\end{align}
Hence, by combining~\cref{proof:quadratic-growth-III-result-1} and~\cref{proof:quadratic-growth-III-result-2}, we obtain
 \begin{align*}
  H^{k+1} - H^\star 
  \leq \left(1 - \frac{\kappa\beta_1}{8 L_1} \right) \left(1 - \frac{\kappa\beta_2}{8 L_2} \right) \left( H^{k} - H^\star \right),
 \end{align*}
and the assertion follows by induction.
\end{proof}

\section{Plain convex case}\label{section:convex}

In this section, sublinear convergence is established for the alternating minimization applied to model problem~\eqref{model-problem} under no additional convexity or growth assumptions, besides plain convexity. A similar setting has been considered by Beck~\cite{Beck2015}. Here, we extend the result in the aforementioned work to Banach spaces, without the use of proximal mappings. As in~\cite{Beck2015}, we however assume a compact level set with respect to the initial value: 

\begin{itemize}
 \item[(P8c)] The functions $g_1:\mathcal{B}_1 \rightarrow \mathbb{R} \cup \{ \infty\}$, $g_2: \mathcal{B}_2 \rightarrow \mathbb{R} \cup \{ \infty \}$ are closed convex (and thereby also $H$ is closed convex). Furthermore, assume that the level set of $H$ with respect to the initial guess
 \begin{align*}
  S(x_0):= \left\{ x \in \mathrm{dom}\, g_1 \times \mathrm{dom}\, g_2 \, \big| \, H(x) \leq H(x^0) \right\}
 \end{align*}
 is compact, and we denote by $R$ the following diameter
 \begin{align*}
  R:=R(x^0):= \mathrm{sup}\left\{ \|x - x^\star \| \, \big| \, x\in S(x^0),\ x^\star \in X \right\}.
 \end{align*}
 By monotonicity of $\{H(x^k)\}_{k=0,\frac{1}{2},1,...}$, it in particular holds
 \begin{align*}
  \| x^k - \bar{x}^k \|, \| x^{k+1/2} - \bar{x}^{k+1/2} \| \leq R,\quad\text{for every }k\geq 0.
 \end{align*}
\end{itemize}

The following convergence result predicts a two-stage behavior: first, the error decreases q-linearly until sufficiently small; after that, sublinear convergence is initiated. The shift is essentially depending on the smoothness properties of the objective function.

\begin{theorem}[Sublinear convergence for the non-smooth convex case]\label{theorem:convergence-convex-nonsmooth}
 Assume that $\mathrm{(P1)}$--$\mathrm{(P7)}$ and $\mathrm{(P8c)}$ are satisfied. Let $\{x^k\}_{k\geq 0}$ be the sequence generated by the alternating minimization, cf.\ \cref{algorithm:alternating-minimization}, and $H^k:=H(x^k)$. Define 
 \begin{align*}
 m^\star :=\left[-1 + \left\lceil\mathrm{log}_2 \left( \frac{H^0 - H^\star}{\mathrm{min}\left\{\frac{L_1}{\beta_1},\frac{L_2}{\beta_2}\right\} R^2} \right)\right\rceil \right]_+,\qquad
 p^\star:= \frac{2\left( \frac{\beta_1}{L_1} + \frac{\beta_2}{L_2} \right)^{-1}}{\mathrm{min}\left\{\frac{L_1}{\beta_1},\frac{L_2}{\beta_2}\right\}}\in[1,2],
 \end{align*}
  where $\lceil\cdot\rceil$ denotes the ceiling function, and $[\cdot]_+$ denotes the positive part of values in $\mathbb{R}$, i.e., $[x]_+:=\mathrm{max}\{x,0\}$.
 Then it holds for all $k\geq 0$
 \begin{align*}
  H^k - H^\star \leq \mathrm{max}\left\{ \left(\frac{1}{2} \right)^k \left(H^0 - H^\star\right), \ \frac{4 R^2  \left( \frac{\beta_1}{L_1} + \frac{\beta_2}{L_2} \right)^{-1}}{[k - m^\star]_+ + p^\star} \right\}.
 \end{align*}
 In particular, for all $k \geq m^\star$, it holds that
 \begin{align*}
  H^k - H^\star \leq \frac{4 R^2  \left( \frac{\beta_1}{L_1} + \frac{\beta_2}{L_2} \right)^{-1}}{k - m^\star + p^\star}.
 \end{align*}
\end{theorem}

The proof of \cref{theorem:convergence-convex-nonsmooth} utilizes two auxiliary results: general descent properties for each subiteration of the alternating minimization, and a criterion for concluding sublinear convergence. Those are summarized in the following two lemmas. 

\begin{lemma}\label{lemma:general-descent-am}
 Under the same assumption as in \cref{theorem:convergence-convex-nonsmooth}, it holds for all $k\geq 0$
 \begin{align}
 \label{result:convergence-convex-nonsmooth-descent-lemma-1}
  H^k - H^{k+1/2} \geq \left\{ \begin{tabular}{ll}
                                $\frac{1}{2} \left( H^k - H^\star \right)$ &if $H^k - H^\star > \frac{2L_1R^2}{\beta_1}$\\
                                $\frac{\beta_1}{4 L_1 R^2} \left( H^k - H^\star \right)^2$ &\text{else},
                               \end{tabular}
\right.
 \end{align}
 and
 \begin{align}
 \label{result:convergence-convex-nonsmooth-descent-lemma-2}
  H^{k+1/2} - H^{k+1} \geq \left\{ \begin{tabular}{ll}
                                $\frac{1}{2} \left( H^{k+1/2} - H^\star \right)$ &if $H^{k+1/2} - H^\star > \frac{2L_2 R^2}{\beta_2}$\\
                                $\frac{\beta_2}{4 L_2 R^2} \left( H^{k+1/2} - H^\star \right)^2$ &\text{else}.
                               \end{tabular}
\right.
 \end{align}
\end{lemma}

\begin{proof}
 We consider the first half step of the alternating minimization, assuming, without loss of generality, that $\tfrac{L_1}{\beta_1}<\infty$; otherwise~\cref{result:convergence-convex-nonsmooth-descent-lemma-1} follows immediately. Following the proof of \cref{theorem:convergence-quadratic-growth-III},~\cref{proof:quadratic-growth-III:aux-4} also holds under the stated assumptions of this lemma. We recall~\cref{proof:quadratic-growth-III:aux-4}: it holds for all $\gamma\in(0,1]$ that
\begin{align*}
 H^k - H^\star
 &\leq 
  \frac{L_1}{\beta_1} \gamma \| x^k - \bar{x}^k \|^2 
  + \frac{1}{\gamma} \left[ H(x_1^k,x_2^k) - H\left(x_1^k + \gamma \left( \bar{x}_1^k - x_1^k \right), x_2^k \right) \right].
\end{align*}
Thus, by the definition of $R$ and $x^{k+1/2}$, cf.~\cref{alternating-minimization-step-1}, it follows
\begin{align*}
 H^k - H^\star \leq
 \frac{L_1R^2}{\beta_1} \gamma + \frac{1}{\gamma} \left( H^k - H^{k+1/2} \right),
\end{align*}
In the following, we distinguish the two cases: (i) $H^k - H^\star > \frac{2 L_1 R^2}{\beta_1}$; and (ii) $H^k - H^\star \leq \frac{2 L_1 R^2}{\beta_1}$. In the first case, we choose $\gamma=1$; in the second case, we choose $\gamma=\frac{\beta_1}{2L_1 R^2}(H^k - H^\star)$. This results in~\cref{result:convergence-convex-nonsmooth-descent-lemma-1}. The result~\cref{result:convergence-convex-nonsmooth-descent-lemma-2} analogously follows by symmetry.
\end{proof}

The following auxiliary lemma is inspired by~\cite{Beck2013}. Opposing to the aforementioned work, the subsequent results allows for effectively making use of the energy descent of both substeps of the alternating minimization instead of just one.

\begin{lemma}\label{lemma:technical-lemma-2}
 Let $\{A_k\}_{k=0,\frac{1}{2},1,\frac{3}{2},...} \subset \mathbb{R}_{>0}$, $\gamma_1,\gamma_2\geq 0$, and $p\geq 0$ satisfying for all $k\geq 0$
 \begin{align}
 \label{result:auxiliary-lemma-II-1}
  A_{k} - A_{k+1/2} &\geq \gamma_1 A_{k}^2,\\
  \label{result:auxiliary-lemma-II-2}
  A_{k+1/2} - A_{k+1} &\geq \gamma_1 A_{k+1/2}^2,
 \end{align}
 and
 \begin{align}
 \label{result:auxiliary-lemma-II-3}
  A_0 \leq \frac{1}{p (\gamma_1 + \gamma_2)}.
 \end{align}
 Then it holds for all $k\geq 0$
 \begin{align*}
  A_{k} \leq \frac{1}{(k+p) (\gamma_1 + \gamma_2)}.
 \end{align*}
\end{lemma}

\begin{proof}
 By~\cref{result:auxiliary-lemma-II-1,result:auxiliary-lemma-II-2}, $\{A_k\}_{k=0,\frac{1}{2},1,\frac{3}{2},...}$ is sequence of decreasing positive numbers. All in all, it holds for $k\geq 0$
 \begin{align*}
  \frac{1}{A_{k+1}} - \frac{1}{A_{k}}
  &=
   \frac{1}{A_{k+1/2}} - \frac{1}{A_{k}} + \frac{1}{A_{k+1}} - \frac{1}{A_{k+1/2}}\\
  &=
  \frac{A_{k} - A_{k+1/2}}{A_{k}A_{k+1/2}} + \frac{A_{k+1/2} - A_{k+1}}{A_{k+1/2}A_{k+1}} \\
  &\geq 
  \gamma_1 \frac{A_{k}}{A_{k+1/2}} + \gamma_2 \frac{A_{k+1/2}}{A_{k+1}} \\
  &\geq 
  \gamma_1 + \gamma_2.
 \end{align*}
 Thus, by utilizing a telescope sum and applying~\cref{result:auxiliary-lemma-II-3}, we obtain
 \begin{align*}
  \frac{1}{A_{k+1}} 
  &=
  \left(\frac{1}{A_{k+1}} - \frac{1}{A_{k}} \right)
  +
  \left(\frac{1}{A_{k}} - \frac{1}{A_{k-1}} \right)
  +
  ...
  +
  \left(\frac{1}{A_{1}} - \frac{1}{A_{0}} \right)
  +
  \frac{1}{A_0} \\
  &\geq 
  (k+1+p) (\gamma_1 + \gamma_2).
 \end{align*}
 Finally, the assertion (for $k\geq 1$) follows by inverting the inequality; for $k=0$ the assertion is identical with the assumption~\eqref{result:auxiliary-lemma-II-3}.
\end{proof}

Finally, we are able to prove \cref{theorem:convergence-convex-nonsmooth}.

\begin{proof}[Proof of \cref{theorem:convergence-convex-nonsmooth}]
  As long as $H^k - H^\star > 2 \,\mathrm{min}\left\{ \tfrac{L_1}{\beta_1}, \tfrac{L_2}{\beta_2} \right\} R^2$ for some $k\in \mathbb{N}_0$, by \cref{lemma:general-descent-am} and the monotonicity of $\{H^k\}_{k=0,1,...}$, it holds that
 \begin{align}
 \label{proof:convergence-convex-nonsmooth-aux:1}
  H^{k} - H^\star \leq \left(\frac{1}{2}\right)^k \left(H^0 - H^\star\right).
 \end{align}
 Thereby, there exists a minimal $m\geq 0$ such that $H^k - H^\star \leq 2 \,\mathrm{min}\left\{ \tfrac{L_1}{\beta_1}, \tfrac{L_2}{\beta_2} \right\} R^2$ for all $k\geq m$. Assuming $m\geq 1$,~\cref{proof:convergence-convex-nonsmooth-aux:1} holds for all $k\leq m-1$. In particular, it holds
 \begin{align*}
  2 \,\mathrm{min}\left\{ \tfrac{L_1}{\beta_1}, \tfrac{L_2}{\beta_2} \right\} R^2 < H^{m-1} - H^\star \leq \frac{1}{2^{m-1}} \left(H^0 - H^\star\right).
 \end{align*}
 Thus, it holds that
 \begin{align*}
  m < \mathrm{log}_2 \left( \frac{H^0 - H^\star}{\,\mathrm{min}\left\{ \tfrac{L_1}{\beta_1}, \tfrac{L_2}{\beta_2} \right\} R^2} \right).
 \end{align*}
 Consequently, in general (including the case $m=0$), it holds
 \begin{align}
 \label{proof:convergence-convex-nonsmooth-aux:2}
  m \leq \left[-1 + \left\lceil\mathrm{log}_2 \left( \frac{H^0 - H^\star}{\,\mathrm{min}\left\{ \tfrac{L_1}{\beta_1}, \tfrac{L_2}{\beta_2} \right\} R^2} \right)\right\rceil \right]_+=:m^\star.
 \end{align}
 
 By the monotonicity of $\{H^k\}_{k=0,\frac{1}{2},1,...}$, there holds $H^{k+1/2} - H^\star \leq H^k - H^\star \leq  2 \,\mathrm{min}\left\{ \tfrac{L_1}{\beta_1}, \tfrac{L_2}{\beta_2} \right\} R^2$ for $k\geq m$. Hence, by \cref{lemma:general-descent-am} it holds for all $k\geq m$ that
 \begin{align*}
  H^k - H^{k+1/2} &\geq \frac{\beta_1}{4 L_1 R^2} \left( H^k - H^\star \right)^2,\\ 
  H^{k+1/2} - H^{k+1} &\geq \frac{\beta_2}{4 L_2 R^2} \left( H^{k+1/2} - H^\star \right)^2.
 \end{align*}
 We define the sequence $\{A_n\}_{n=0,\frac{1}{2},1,...}$ with $A_n := H^{n+m} - H^\star$. Then the assumptions of \cref{lemma:technical-lemma-2} are fulfilled with
 \begin{align*}
  \gamma_1 := \frac{\beta_1}{4 L_1 R^2},\quad \gamma_2:= \frac{\beta_2}{4 L_2 R^2},\quad p:=
  \frac{2\left( \frac{\beta_1}{L_1} + \frac{\beta_2}{L_2} \right)^{-1}}{\,\mathrm{min}\left\{ \tfrac{L_1}{\beta_1}, \tfrac{L_2}{\beta_2} \right\}}=:p^\star.
 \end{align*}
 Thus, \cref{lemma:technical-lemma-2} yields for all $n\geq 0$
 \begin{align*}
  A_n \leq \frac{1}{\left(n + p \right) \left(\gamma_1 + \gamma_2\right)},
 \end{align*}
 and equivalently for all $k\geq m$
 \begin{align}
 \label{proof:convergence-convex-nonsmooth-aux:3}
  H^k - H^\star \leq \frac{4 R^2  \left( \frac{\beta_1}{L_1} + \frac{\beta_2}{L_2} \right)^{-1}}{k - m + p^\star}
  \leq \frac{4 R^2  \left( \frac{\beta_1}{L_1} + \frac{\beta_2}{L_2} \right)^{-1}}{[k - m^\star]_+ + p^\star}
 \end{align}
 Combining~\cref{proof:convergence-convex-nonsmooth-aux:1,proof:convergence-convex-nonsmooth-aux:3} proves the assertion.
\end{proof}

\begin{remark}
 In the case it holds $\,\mathrm{max}\left\{ \tfrac{L_1}{\beta_1}, \tfrac{L_2}{\beta_2} \right\}<\infty$, and the initial error satisfies $H^{0}-H^\star > 2 \,\mathrm{max}\left\{ \tfrac{L_1}{\beta_1}, \tfrac{L_2}{\beta_2} \right\} R^2$, the result of \cref{theorem:convergence-convex-nonsmooth} is in fact slightly pessimistic. The value of $m^\star$ can then be chosen significantly lower. By an analogous line of argumentation as in the above proof, one can conclude that,  $H^k - H^\star$ first contracts with a rate of $\frac{1}{4}$ for the first $k_1$ iterations, until $H^{k_1}-H^\star\leq 2 \,\mathrm{max}\left\{ \tfrac{L_1}{\beta_1}, \tfrac{L_2}{\beta_2} \right\} R^2$ for some $k_1\in\mathbb{N}_0$. Afterwards, the convergence behavior can be qualitatively predicted as in \cref{theorem:convergence-convex-nonsmooth}. Ultimately, $m^\star$ takes a lower value of the order
 \begin{align*}
  m^\star \approx \left\lceil \mathrm{log}_4 \left(\frac{H^0 - H^\star}{2 \,\mathrm{max}\left\{ \tfrac{L_1}{\beta_1}, \tfrac{L_2}{\beta_2} \right\} R^2}\right) \right\rceil + \left\lceil \mathrm{log}_2 \left( \frac{\mathrm{max}\left\{ \tfrac{L_1}{\beta_1}, \tfrac{L_2}{\beta_2} \right\}}{\mathrm{min}\left\{ \tfrac{L_1}{\beta_1}, \tfrac{L_2}{\beta_2} \right\}} \right) \right\rceil.
 \end{align*} 
 However, for the sake of a cleaner presentation, we have avoided the anyhow rather theoretical accurate description of the convergence for the general case.
\end{remark}

\begin{remark}[Comparison to the literature]
 Beck~\cite{Beck2015} establishes sublinear convergence for the alternating minimization applied to the model problem~\eqref{model-problem}, where only finite dimensional Euclidean spaces are considered equipped with $l_2$ norms; i.e., it holds $\beta_1=\beta_2=1$. Utilizing knowledge on proximal mappings with respect to the two blocks and associated sufficient decrease properties, the final results reads: For $k\geq 2$ it holds that
 \begin{align*}
  H^k - H^\star \leq \mathrm{max}\left\{ \left(\frac{1}{2} \right)^\frac{k-1}{2} \left(H^0 - H^\star \right), \frac{8\,\mathrm{min}\{L_1,L_2\}R^2}{k-1} \right\}.
 \end{align*}
 By only employing convexity and smoothness properties of the problem, our result also holds for general Banach spaces. Furthermore, focussing on the sublinear convergence, our result gains from the use of both $L_1$ and $L_2$, resulting in a potentially slightly lower multiplicative constant.
\end{remark}

\subsection{Smooth case in Euclidean spaces}
 
The tools, established in the previous section, allow for the improvement of a convergence result by Beck and Tetruashvili~\cite{Beck2013} on the alternating minimization applied to the smooth model problem~\eqref{model-problem} in a Euclidean setting. 

In the following, we consider Euclidean spaces for $\mathcal{B}_1$ and $\mathcal{B}_2$, equipped with $l_2$ norms, i.e., it holds $\beta_1=\beta_2=1$. Furthermore, let $g_1 \equiv g_2 \equiv 0$, ensuring smoothness of the model problem~\eqref{model-problem}. For this setting, sublinear convergence has been established in~\cite{Beck2013}, stating that it holds for all $k\geq 2$
 \begin{align*}
 H^k - H^\star \leq \frac{2\, \mathrm{min}\{ L_1,L_2\}R^2}{k-1}.
\end{align*}
In the following, we show that in fact for all $k\geq 0$ it holds in fact the improved result
 \begin{align}
 \label{result:sublinear-convergence-smooth-convex}
 H^k - H^\star \leq \frac{2 \left( \tfrac{1}{L_1} + \tfrac{1}{L_2} \right)^{-1}R^2 }{k + 2 R^2 \left( \tfrac{1}{L_1} + \tfrac{1}{L_2} \right)^{-1}\left( H^0 - H^\star \right)^{-1}} \leq \frac{2 \left( \tfrac{1}{L_1} + \tfrac{1}{L_2} \right)^{-1} R^2}{k}.
\end{align}
Thus, not only the subproblem with lower Lipschitz constant governs the performance of the alternating minimization. But the performance separately benefits from both Lipschitz constants.

For deriving~\cref{result:sublinear-convergence-smooth-convex}, we combine descent properties of the alternating minimization derived in~\cite{Beck2013}, and the auxiliary \cref{lemma:technical-lemma-2}. The following descent properties are a byproduct of the proof of Theorem~5.2 in~\cite{Beck2013}.

\begin{lemma}[Descent properties of the alternating minimization for the smooth case~\cite{Beck2013}]\label{lemma:descent-beck-2013}
Assume $\mathrm{(P1)}$--$\mathrm{(P7)}$ and $\mathrm{(P8c)}$ are satisfied with $\mathcal{B}_1$ and $\mathcal{B}_2$ being Euclidean spaces, equipped with $l_2$ norms, and $g_1 \equiv g_2 \equiv 0$. Furthermore, let $\{x^k\}_{k\geq 0}$ be the sequence generated by the alternating minimization, cf.\ \cref{algorithm:alternating-minimization}, and $H^k:=H(x^k)$. Then it holds for all $k\geq 0$
 \begin{align*}
  H^k - H^{k+1/2} &\geq \frac{1}{2L_1 R^2} \left(H^{k} - H^\star\right)^2, \\
  H^{k+1/2} - H^{k+1} &\geq \frac{1}{2 L_2 R^2} \left(H^{k+1/2} - H^\star\right)^2.
 \end{align*}
\end{lemma}
\noindent
Finally, by \cref{lemma:descent-beck-2013}, the assumptions of \cref{lemma:technical-lemma-2} are satisfied for $A_{k}:= H^k - H^\star$, and 
\begin{align*}
 \gamma_1 := \frac{1}{2L_1 R^2},\quad \gamma_2:= \frac{1}{2L_2R^2},\quad p:=\frac{1}{\left(\gamma_1 + \gamma_2\right) A_0} = \frac{2 R^2 \left( \tfrac{1}{L_1} + \tfrac{1}{L_2} \right)^{-1}}{ H^0 - H^\star}.
\end{align*}
Thus, the sublinear convergence result~\eqref{result:sublinear-convergence-smooth-convex} directly follows from \cref{lemma:technical-lemma-2}.

\section{Conclusions}
\label{sec:conclusions} 

In this paper, we established convergence of the alternating minimization applied to a two-block structured model problem within the class of non-smooth non-strongly convex optimization in Banach spaces. We considered three cases of relaxed strong convexity: quasi-strong convexity, quadratic functional growth, and plain convexity. Convergence rates have been provided -- of linear type for the first two cases, and of sublinear type for the third case. Opposing to previous works on the convergence analysis of the alternating minimization, we have considered a fairly general setup. Ultimately, by allowing to describe smoothness and convexity properties with respect to different norms, improved convergence rates have been derived in comparison to corresponding results in the literature~\cite{Beck2013,Beck2015}. In particular, the linear convergence  in the case of quadratic functional growth also holds without any feasible descent property as commonly required in the analysis of the general block coordinate descent~\cite{Luo1993,Necoara2019}. 

Our results have several implications. For instance, applications of the results in~\cite{Beck2013,Beck2015} can be immediately improved, e.g., for the iteratively reweighted least squares method~\cite{Beck2015}. Also the tools provided in this paper allow for a sharp problem-specific convergence analysis of iterative splitting schemes for coupled partial differential equations; this has been realized with similar but slightly simpler abstract convergence results in~\cite{Both2019}. 

Within the proofs of this work, it was never used that the norms $\|\cdot\|_1$, $\|\cdot \|_2$, and $\|\cdot\|$ are actually norms. The results could be directly relaxed measuring convexity and smoothness with respect to just semi-norms or something similar. This may allow to generalize our results even further. The same applies for the general block coordinate descent method.
 
\bibliographystyle{siamplain}

\begin{thebibliography}{100}

\bibitem{Auslender1976}
{\sc A.~Auslender}, {\em Optimisation}, Masson, Paris, 1999.

\bibitem{Beck2015}
{\sc A.~Beck}, {\em On the convergence of alternating minimization for convex
  programming with applications to iteratively reweighted least squares and
  decomposition schemes}, SIAM Journal on Optimization, 25 (2015),
  pp.~185--209, doi:10.1137/13094829X.

\bibitem{Beck2013}
{\sc A.~Beck and L.~Tetruashvili}, {\em On the convergence of block coordinate
  descent type methods}, SIAM Journal on Optimization, 23 (2013),
  pp.~2037--2060, doi:10.1137/120887679.

\bibitem{Bertsekas1999}
{\sc D.~Bertsekas}, {\em Nonlinear Programming}, Athena scientific optimization
  and computation series, Athena Scientific, 1999.

\bibitem{Both2019}
{\sc J.~W. {Both}, K.~{Kumar}, J.~M. {Nordbotten}, and F.~A. {Radu}}, {\em {The
  gradient flow structures of thermo-poro-visco-elastic processes in porous
  media}}, arXiv e-prints,  (2019), arXiv:1907.03134.

\bibitem{Grippo2000}
{\sc L.~Grippo and M.~Sciandrone}, {\em On the convergence of the block
  nonlinear gauss--seidel method under convex constraints}, Operations research
  letters, 26 (2000), pp.~127--136.

\bibitem{Grippo1999}
{\sc L.~Grippof and M.~Sciandrone}, {\em Globally convergent block-coordinate
  techniques for unconstrained optimization}, Optimization Methods and
  Software, 10 (1999), pp.~587--637,
  doi:10.1080/10556789908805730.

\bibitem{Luo1993}
{\sc Z.-Q. Luo and P.~Tseng}, {\em Error bounds and convergence analysis of
  feasible descent methods: a general approach}, Annals of Operations Research,
  46 (1993), pp.~157--178, doi:10.1007/BF02096261.

\bibitem{Mordukhovich2006}
{\sc B.~S. Mordukhovich, N.~M. Nam, and N.~D. Yen}, {\em Fréchet
  subdifferential calculus and optimality conditions in nondifferentiable
  programming}, Optimization, 55 (2006), pp.~685--708,
  doi:10.1080/02331930600816395.

\bibitem{Necoara2019}
{\sc I.~Necoara, Y.~Nesterov, and F.~Glineur}, {\em Linear convergence of first
  order methods for non-strongly convex optimization}, Math. Program., 175
  (2019), pp.~69--107, doi:10.1007/s10107-018-1232-1.

\bibitem{Ortega1970}
{\sc J.~M. Ortega and W.~C. Rheinboldt}, {\em Iterative solution of nonlinear
  equations in several variables}, vol.~30, SIAM, 1970.

\bibitem{Tai1998}
{\sc X.-C. Tai and M.~Espedal}, {\em {Rate {O}f {C}onvergence {O}f {S}ome
  {S}pace {D}ecomposition {M}ethods {F}or {L}inear {A}nd {N}onlinear
  {P}roblems}}, SIAM J. Numer. Anal, 35 (1998), pp.~1558--1570.

\bibitem{Wang2014}
{\sc P.-W. Wang and C.-J. Lin}, {\em Iteration complexity of feasible descent
  methods for convex optimization}, Journal of Machine Learning Research, 15
  (2014), pp.~1523--1548.

\bibitem{Wright2015}
{\sc S.~J. Wright}, {\em Coordinate descent algorithms}, Mathematical
  Programming, 151 (2015), pp.~3--34,
  doi:10.1007/s10107-015-0892-3.

\bibitem{Zhang2015}
{\sc H.~Zhang and L.~Cheng}, {\em Restricted strong convexity and its
  applications to convergence analysis of gradient-type methods in convex
  optimization}, Optimization Letters, 9 (2015), pp.~961--979,
  doi:10.1007/s11590-014-0795-x.

\end{thebibliography}

\end{document}